\documentclass[11pt]{article}

\usepackage{url}

\usepackage{amsmath}
\usepackage{amsthm}
\usepackage{amssymb}
\usepackage{latexsym}
\usepackage{pstricks}
\usepackage{graphicx, pst-plot, pst-node, pst-text, pst-tree, young, ytableau}
\usepackage{titlefoot}
\usepackage{enumerate}
\usepackage{titlesec}
\usepackage{cancel}
\usepackage{rotating}
\usepackage[small,it]{caption}
\usepackage{color}
\definecolor{lightgray}{rgb}{0.8, 0.8, 0.8}
\definecolor{darkgray}{rgb}{0.7, 0.7, 0.7}
\definecolor{darkblue}{rgb}{0, 0, .4}

\usepackage[bookmarks]{hyperref}
\hypersetup{
        colorlinks=true,
        linkcolor=darkblue,
        anchorcolor=darkblue,
        citecolor=darkblue,
        urlcolor=darkblue,
        pdfpagemode=UseThumbs,
        pdftitle={Stuff on 312 avoiding squares},
        pdfsubject={Combinatorics},
        pdfauthor={Miklo\'s Bo\'na and Rebecca Smith },
}

\newtheorem{theorem}{Theorem}[section]
\newtheorem{proposition}[theorem]{Proposition}

\newtheorem{lemma}[theorem]{Lemma}

\newtheorem{corollary}[theorem]{Corollary}
\newtheorem{conjecture}[theorem]{Conjecture}
\newtheorem{question}[theorem]{Question}
\newtheorem{remark}[theorem]{Remark}

\newtheoremstyle{example}{\topsep}{\topsep}%
     {}
     {}
     {\bfseries}
     {.}
     {.5em}
     {\thmname{#1}\thmnumber{ #2}}
\theoremstyle{example}
\newtheorem{example}[theorem]{Example}

\newtheoremstyle{negexample}{\topsep}{\topsep}%
     {}
     {}
     {\bfseries}
     {.}
     {.5em}
     {\thmname{#1}\thmnumber{ #2}}
\theoremstyle{negexample}

\newcounter{todocounter}

\long\def\symbolfootnote[#1]#2{\begingroup%
\def\thefootnote{\fnsymbol{footnote}}\footnote[#1]{#2}\endgroup}

\makeatletter
\newcommand{\Rm}[1]{\expandafter\@slowromancap\romannumeral #1@}
\newfont{\footsc}{cmcsc10 at 8truept}
\newfont{\footbf}{cmbx10 at 8truept}
\newfont{\footrm}{cmr10 at 10truept}
\renewenvironment{abstract}%
                {
                  \begin{list}{}%
                     {\setlength{\rightmargin}{1in}%
                      \setlength{\leftmargin}{1in}}%
                   \item[]\ignorespaces\begin{small}}%
                 {\end{small}\unskip\end{list}}

\amssubj{05A17, 06A07}
\keywords{cycles, permutation pattern, permutation squares}

\newpagestyle{main}[\small]{
        \headrule
        \sethead[\usepage][][]
        {\sc 312-avoiders and squares}{}{\usepage}}

\title{\sc{Pattern avoidance in permutations and their squares}}
\author{Mikl\'os B\'ona
\\[-0.25ex]
\small Department of Mathematics\\[-0.5ex]
\small University of Florida\\[-0.5ex]
\small Gainesville, Florida\\[15pt]
Rebecca Smith
\\[-0.25ex]
\small Department of Mathematics\\[-0.5ex]
\small SUNY Brockport\\[-0.5ex]
\small Brockport, New York\\[-1.5ex]
}

\titleformat{\section}
        {\large\sc}
        {\thesection.}{1em}{}   

\date{}

\begin{document}
\maketitle

\newcommand{\s}{\mathbf{s}}
\newcommand{\m}{\mathbf{m}}
\renewcommand{\t}{\mathbf{t}}
\renewcommand{\b}{\mathbf{b}}
\newcommand{\f}{\mathbf{f}}
\newcommand{\rev}{\operatorname{rev}}
\newcommand{\dual}{\operatorname{dual}}
\newcommand{\D}{\mathcal{D}}
\newcommand{\Av}{\operatorname{Av}}

\newcommand{\inp}{\textsf{i}}
\newcommand{\tra}{\textsf{t}}
\newcommand{\out}{\textsf{o}}

\newcommand{\R}{\stackrel{R}{\sim}}
\renewcommand{\L}{\stackrel{L}{\sim}}
\newcommand{\notR}{\stackrel{R}{\not\sim}}
\newcommand{\notL}{\stackrel{L}{\not\sim}}

\newcommand{\OEISlink}[1]{\href{https://urldefense.proofpoint.com/v2/url?u=http-3A__oeis.org_-231-257D-257B-231&d=DwIGAg&c=sJ6xIWYx-zLMB3EPkvcnVg&r=0Htk6_We6KpopYAypczNQg&m=QsVkuMFt81EDlbnFHQ2lUp7u6ePHdD4JErAIQ1Aj-Tc&s=I1IDjLEK-ges6548LanY6fR37FvIzBfjQwcFXZye1yY&e= }}
\newcommand{\OEISref}{\href{https://urldefense.proofpoint.com/v2/url?u=http-3A__oeis.org_-257D-257BOEIS-257D-7E-255Ccite-257Bsloane-3Athe-2Don-2Dline-2Denc&d=DwIGAg&c=sJ6xIWYx-zLMB3EPkvcnVg&r=0Htk6_We6KpopYAypczNQg&m=QsVkuMFt81EDlbnFHQ2lUp7u6ePHdD4JErAIQ1Aj-Tc&s=TdYfuwERNqAex_sRNePm3XCj8Qsdv0d5jqfwODPRWGI&e= :}}
\newcommand{\OEIS}[1]{(Sequence \OEISlink{#1} in the \OEISref.)}

%
%
%
%
%
%
%
%

\def\sdwys #1{\xHyphenate#1$\wholeString}
\def\xHyphenate#1#2\wholeString {\if#1$%
\else\say{\ensuremath{#1}}\hspace{2pt}%
\takeTheRest#2\ofTheString
\fi}
\def\takeTheRest#1\ofTheString\fi
{\fi \xHyphenate#1\wholeString}
\def\say#1{\begin{turn}{-90}\ensuremath{#1}\end{turn}}

\begin{abstract} We study permutations $p$ such that both $p$ and $p^2$ avoid a given pattern $q$. 
We obtain a generating function for the case of $q=312$ (equivalently, $q=231$), we prove that if $q$ is monotone increasing, then above a certain length, there are no such permutations, and we prove an upper
bound for $q=321$. We also present some intriguing questions in the case of $q=132$.  
\end{abstract}

\section{Introduction}
The vast and prolific area of {\em permutation patterns} considers permutations as {\em linear orders}. In this concept,
a permutation $p$ is a linear array $p_1p_2\cdots p_n$ of the first $n$ positive integers in some order, so that each 
integer occurs exactly once. We say that a permutation $p$ {\em contains} the pattern $q=q_1q_2\cdots q_k$ 
if there is a $k$-element set of indices $i_1<i_2< \cdots <i_k$ so that $p_{i_r} < p_{i_s} $ if and only
if $q_r<q_s$.  If $p$ does not contain $q$, then we say that $p$ {\em avoids} $q$. A recent survey on permutation 
patterns can be found in \cite{vatter} and a book on the subject is \cite{combperm}. 

This definition does not consider the other perspective from which permutations can be studied, namely that of
the {\em symmetric group}, where the product of two  permutations is defined, and the notion of a permutation's inverse
is defined. Therefore, it is not surprising that pattern avoidance questions become much more difficult if the symmetric
group concept is present in them. (See \cite{archer}, \cite{cory} or \cite{huang} for a few results in this direction.)
 One exception to this is the straightforward observation \cite{awalk} that if $p$ avoids $q$, 
then its inverse permutation $p^{-1}$ avoids $q^{-1}$. 

Trying to extend that simple observation at least a little bit, in this paper we study the following family of questions. 
Let us call a permutation $p$ {\em strongly $q$-avoiding} if both $p$ and $p^2$  avoid $q$. Let $\textup{Sav}_n(q)$ denote the number of strongly $q$-avoiding permutations of length $n$. What can be said about the numbers 
$\textup{Sav}_n(q)$? As far as we know, this is the first time that questions of this type are considered, that is, 
when one attempts to enumerate permutations $p$ so that both $p$ and $p^2$ are to avoid the {\em same}
pattern $q$. 

In Section \ref{sec-monotone}, we will prove that if $q$ is a monotone increasing pattern of any length, then 
$\textup{Sav}_n(q)=0$ will hold for
$n$ sufficiently large. This is only true for monotone increasing patterns. Then in Section  \ref{312-george_1},
we will  provide an explicit formula for
the generating function for the sequence of numbers $\textup{Sav}_n(312)=\textup{Sav}_n(231)$, while in 
Section \ref{sec-321}, we will prove a lower
bound for the sequence $\textup{Sav}_n(321)$.  Finally, in Section \ref{sec132}, we end the paper with intriguing questions about the (identical) sequences
 $\textup{Sav}_n(132)=\textup{Sav}_n(213)$.

\section{The pattern $12\cdots k$} \label{sec-monotone}

In this section, we prove that for all $k$,  if $p$ is long enough, then either $p$ or $p^2$ must contain an increasing subsequence of length $k$.

\begin{theorem}~\label{monotone_bound} Let $k$ be a positive integer, and let $n\geq (k-1)^3+1$. Then $\textup{Sav}_n(12\cdots k)=0$.
\end{theorem}

\begin{proof} Let $p$ be a permutation of length $n\geq  (k-1)^3+1$ that avoids the pattern $12\cdots k$.
Then $p$ is the union of $k-1$ decreasing subsequences. Color these subsequences with colors $1,2,\ldots ,k-1$ so that
a maximum-length subsequence gets the color 1. Then there are at least $(k-1)^2+1$ entries of color 1. Let us say that these entries are in positions $i_1<i_2<\cdots <i_m$, and they are $p(i_1)>p(i_2)>\cdots >p(i_m)$. In order to simplify 
notation, let $P_j=p(i_j)$.

 Now consider the $m\geq (k-1)^2+1$
positions $P_m<P_{m-1}<\cdots <P_1$ in $p$. By the Pigeon-hole principle, there will be a set $S$ of at least $k$ positions
among them so that entries
in these positions  are of the same color, that is, that form a decreasing subsequence $S$. Then we claim 
that those same
$k$ entries form an increasing sequence in $p^2$.  Indeed, let positions $P_a$ and $P_b$ contain two entries of $S$,
with $a<b$, so $P_a>P_b$.
That means that position $p(i_a)$ contains a smaller entry than position $p(i_b)$, that is, $p(p(i_a))<p(p(i_b))$. 
This argument can be repeated for every pair of entries in positions that belong to $S$, proving that in $p^2$, the 
$k$ entries that are in positions in $S$ form an increasing subsequence.
\end{proof}

The bound on the length of $n$ relative to $k$ for a strongly $12\cdots k$-avoiding permutation given above is tight for at least small values of $k$.  This is trivially true for $k=1,2$.  Moreover, we are going to present
a strongly  $12\cdots k$-avoiding permutation of length $(k-1)^3$ for each of $k=3, 4, 5, 6$. In these cases, the strategy is to create $p$ such that
$p^2$ will be a be a specific layered permutation, that is, it will consist of $k-1$ decreasing subsequences (the layers) of 
length $(k-1)^2$ so that the entries decrease within each layer, but increase from one layer to the next.  As layered permutations are involutions, this means each $p$ has order four.  Further since $p^3 = p^{-1}$ and the inverse of the identity permutation is itself, we also have $p^3$ avoiding $12\cdots k$ in these examples.  However, finding a good $p$ (one that avoids $12\cdots k$ while producing the desired $p^2$) makes the construction challenging.

Indeed, for $k=3$,
let  $p=75863142=(1746)(2538)$, then we have $p^2=43218765$.  A similar approach yields a maximum length strongly $12\cdots k$-avoiding permutation of length $(k-1)^3$ for $k=4$:
\begin{align*}
p &= 24 \; 21 \; 26 \; 19 \; 23 \; 27 \; 20 \; 25 \; 22 \; 15 \; 12 \; 17 \; 10 \; 14 \; 18 \; 11 \; 16 \;13 \; 6 \; 3 \; 8 \; 1 \; 5 \; 9 \; 2 \;7 \; 4 \\
p^2 &= 9 \; 8 \; 7\; 6 \; 5\; 4\; 3\; 2\; 1 \; 18 \;17\;16\;15\;14\;13\;12\;11\;10 \; 27 \; 26 \;25 \; 24 \; 23 \; 22 \; 21 \; 20 \;19
\end{align*}

Somewhat similarly constructed strongly $12\cdots k$-avoiding permutations of length $(k-1)^3$ also exist for $k=5,6$.  

For $k=5$, each of the four intervals of consecutive entries is of the form: 
\[
9 \; 5 \; 11 \; 7 \; 15 \; 3 \; 13 \; 1 \; 16 \; 4 \; 14 \; 2 \; 10 \; 6 \; 12\;  8 = (1\; 9\; 16\; 8)(2\; 5\; 15\; 12)(3\; 11\; 14\; 6)(4\; 7\; 13\; 10).
\]
From left to right, each interval is made up of the largest remaining $16$ entries.

For $k=6$, the five intervals are of the form:
\[
14\; 9\; 18\; 7\; 16\; 11\; 22\; 3\; 24\; 5\; 20\; 1\; 13\; 25\; 6\; 21\; 2\; 23\; 4\; 15\; 10\; 19\; 8\; 17\; 12.
\]
That is, $(1\; 14\; 25\; 12) (2\; 9\; 24\; 17) (3\; 18\; 23\; 8) (4\; 7\; 22\; 19) (5\; 16\; 21\; 10) (6\; 11\; 20\; 15)$.
From left to right, each interval is made up of the largest remaining $25$ entries.


There are also nice graphical symmetries within the intervals in these small examples.  As such, there is reason to believe it could be possible to generalize these constructions to show the bound given in Theorem~\ref{monotone_bound} is always the best.  

\begin{conjecture}  Let $k$ be a positive integer, and let $n = (k-1)^3$. Then $\textup{Sav}_n(12\cdots k) \neq 0$.
\end{conjecture}

We have not formulated a conjecture as to how many such maximal length permutations strongly avoid $12\cdots k$.  Our example is not unique for $k=3$ and (for example $54387612$ also strongly avoids $123$), but we have not investigated this for larger values of $k$. Furthermore, if $p^2$ is layered, and $p$ is strongly 
$12\cdots k$-avoiding, then, as we pointed out, $p^3=p^{-1}$ is also $12\cdots k$-avoiding. Our constructions
for $3\leq k\leq 6$ are examples for this. 

Note that no other patterns $q$ have the property that $\textup{Sav}_n(q)=0$ if $q$ is large enough. Indeed, 
the identity permutation strongly avoids all patterns that are not monotone increasing.

\section{The pattern 312} \label{312-george_1}

In a 312-avoiding permutation, all entries on the left of the entry 1 must be smaller than all entries on the right of 1, or
a 312-pattern would be formed with the entry 1 in the middle. Therefore, if $p=p_1p_2\cdots p_n$ is a 312-avoiding 
permutation, and $p_i=1$, then $p$ maps the interval $[1,i]$ into itself, and the interval $[i+1,n]$ into itself. 
That means that $p$ will be strongly 312-avoiding if and only if its restrictions to those two intervals are strongly
312-avoiding. In other words, each non-empty strongly 312-avoiding permutation $p$  uniquely decomposes as $p=LR$, where
$L$ is a strongly 312-avoiding permutation ending in the entry 1, and $R$ is a (possibly empty) strongly 312-avoiding permutation (on the set $[i+1,n])$). Sometimes, this is described by writing that $p=L\oplus R$, and
saying that $p$ is the direct sum of $L$ and $R$. 

Therefore, if $\textup{Sav}_{312}(z)=\sum_{n\geq 0} \textup{Sav}_n(312)z^n$, and $B(z)$ is the ordinary generating function for
the number of strongly 312-avoiding permutations ending in 1, then the equality
\begin{equation} \label{funceq} \textup{Sav}_{312}(z) = 1+ B(z)\textup{Sav}_{312}(z),\end{equation}
holds.  This motivates our analysis of strongly 312-avoiding permutations that end in the entry 1.

\subsection{Permutations ending in 1}
Our goal in this section is to prove the following theorem that characterizes strongly 312-avoiding permutations that
end in 1.  

\begin{theorem} \label{george_ending_in_1}  For any permutation $p$ ending in $1$, the following two statements are equivalent. 
\begin{enumerate}
\item[(A)] The permutation $p$ is strongly 312-avoiding.
\item[(B)] The permutation $p$ has form $p = (k+1) (k+2) \cdots n \; k (k-1) (k-2) \cdots 1$ where $k \geq \frac{n}{2}$.  That is, $p$ is unimodal beginning with its $n-k \leq \frac{n}{2}$ largest entries in increasing order followed by the remaining $k$ smallest entries in decreasing order.
\end{enumerate}
\end{theorem}

We will prove this theorem through a sequence of lemmas.  In this section, we will assume $p$ is a permutation of length $n$ ending in its entry 1.  
First, we will show that strongly 312-avoiding permutations must end in a long decreasing subsequence.

\begin{lemma} \label{21} Let $p=p_1p_2\cdots p_{n-1}1$ be a strongly 312-avoiding permutation.  If $n \geq 2$, then $p_{n-1}=2$.
\end{lemma}

\begin{proof}  Suppose by way of contradiction that $p_i =2$ where $i \neq n-1$.  Then, because $p$ must avoid 312, we know that $p_a<p_k$ if $a\leq i <k<n$.  Specifically, $p_1<p_k$ if $i<k<n$. Hence $p_{i +j} = n$ for some positive integer $j$.

Now consider $p^2$.  Note first $p^2(i+j) =1$ and $p^2(n) = p_1$.  There are at least \[(n-1) -(i+1)+1=n-i-1\] entries (namely $p_{i+1}, p_{i+2}, \ldots, p_{n-1}$) all of which are larger than $p_1$.  However, there are at most $n-1- (i+2) +1=n-i-2$ positions between $1$ and $p_1$ in $p^2$.  Hence at least one of these large entries appears before the $1$ which creates the forbidden 312 pattern in $p^2$.
\end{proof}

\begin{remark} \label{n_remark}  When $n$ appears in a strongly 312-avoiding permutation $p$, the entries that follow $n$ must be in decreasing order to avoid a 312 pattern in $p$.
\end{remark}

We now extend Lemma~\ref{21} to show a strongly 312-avoiding permutation $p$ must end in a consecutive decreasing sequence of length at least $\frac{n}{2}$.

\begin{lemma} \label{decrease} Let $p=p_1p_2\cdots p_{n-1}1$ be a strongly 312-avoiding permutation ending in 1.  The smallest $\lceil \frac{n}{2} \rceil$ entries of $p$ appear in the last $\lceil \frac{n}{2} \rceil$ positions in decreasing order. \end{lemma}

\begin{proof}  By way of contradiction, let us assume that $p$ is a strongly 312-avoiding permutation, and
let us also assume that the longest decreasing subsequence at the end of $p$ that consists of consecutive integers
in consecutive positions starts with the entry $m < \frac{n}{2}$.   So $m+1$ is the smallest entry where $p$ deviates from the described form.  

Then if $p_i=m+1$, we have $1 \leq i \leq n-m-1$.   As with Lemma~\ref{21}, the entries $p_1, p_2, \ldots, p_{i}$ must all be smaller than the entries $p_{i+1}, p_{i+2}, \ldots, p_{n-m}$, and in particular $p_{i+j}=n$ for some positive integer $j$.  
(Otherwise, $p$ would contain a 312-pattern with $p_i=m+1$ in the middle.)

We now consider $p^2$.  Note first $p^2(i+j) = 1$.  

Notice $p^2(i) = p_{m+1}$.  If $i < m+1$, then $p_{m+1} > p_1$ since $p_{m+1}$ appears to the right of $m+1$ and $m+1 < n-m+1$ so $p_{m+1}$ is not part of our descending sequence of small entries.  This means $p^2(i) p^2(i+j) p^2(n)$ form a 312-pattern in $p^2$.

Otherwise, $i \geq m+1$.  Then $p(r)=i+1$ for some $r <i$.  That is, one of the entries to the left of $m+1$ in $p$ maps to the position of a larger element to the right of $m+1$ in $p$.  Hence, $p^2(r) p^2(i+j) p^2(n)$ form a 312 pattern in $p^2$.

Thus every strongly 312-avoiding permutation $p$ must end with a consecutive decreasing sequence of at least half of its smallest entries.
\end{proof}

Next, we show that except for the monotone decreasing permutation, all strongly 312-avoiding permutations $p$ begin with an ascent.

\begin{lemma}~\label{monotone_dec}  If $p$ is a strongly 312-avoiding permutation ending in 1,  and $p$ begins with a descent, then $p$
is the decreasing permutation $n \; (n - 1) \cdots 3 \; 2 \; 1$.
\end{lemma}

\begin{proof}  Suppose that $p$ does begin with a descent, that is, suppose $p_1 > p_2$.  Then, first notice if $p_1 \neq n$, then in $p^2$, we have $n$ appearing somewhere in the first $n-2$ positions.  However, then $n$ with the last two entries $p_2, p_1$ form a 312-pattern in $p^2$.  Hence if $p$ begins with a descent, then $p_1 = n$.  Now, as mentioned in Remark~\ref{n_remark}, all subsequent entries must appear in descending order to avoid a 312-pattern in $p$.
\end{proof}

By Lemma~\ref{decrease}, we know that a strongly 312-avoiding permutation $p$ that ends in 1 ends with at least a half of its smallest entries forming a decreasing subsequence of consecutive entries.  We now extend this to complete our proof that $p$ must have the unimodal structure described in Theorem~\ref{george_ending_in_1} be a strongly 312-avoiding permutation.

\begin{lemma}  A strongly 312-avoiding permutation $p$ ending in its entry 1
  must be of the form $p = (k+1) (k+2) \cdots n \; k (k-1) (k-2) \cdots 1$.
\end{lemma}

\begin{proof}
Let $k$ be the length of the longest consecutive decreasing sequence at the end of $p$.  Now consider the large entries $k+1, k+2, \ldots, n$.  If $p$ does not satisfy the structure, then at least two of these large entries must form a descent.  Of all these descents, choose the descent $p_ip_{i+1}$ such that $p_{i+1}$ is the smallest.  

We first note that not only must these large entries of $p$ avoid 312, they must also avoid forming a $213$ pattern in $p$.  This is since the large entries appear in the first $\lfloor \frac{n}{2} \rfloor$ positions in $p$ and the entries corresponding to those positions are in decreasing order in $p$ by Lemma~\ref{decrease}.  Hence a $213$ of large entries in $p$ would form a 312 pattern in $p^2$.

%
%

Similarly, the large entries must avoid forming a $132$ in $p$.  If there was such a pattern, say $p_a p_b p_c$, the entries in these position in $p^2$, namely $p^2_a p^2_b p^2_c $ form a $312$ pattern. 

Thus all of $p_{1}, p_{2}, \ldots, p_{i-1}$ must be larger than  $p_{i+1}$ as a smaller such entry along with entries $p_{i}, p_{i+1}$ would form $132$ pattern in $p$.  

Consider the entry $p_{i+2}$.  First note $p_{i+2}$ must be one of the large entries since $p_{i+1} < p_{j}$ for all $j \leq i$, and so $p_{i+2}$ being part of the long decreasing sequence at the end of $p$ would mean $p_{i+1}$ is part of it as well.  Thus $p_{i+2} > p_{i+1}$ by the minimality constraint on the chosen descent.
However, this forces $p_{i} p_{i+1} p_{i+2}$ to form a $213$ or a $312$ among the large entries of $p$, both of which are not allowed.

Hence if $p$ is a strongly 312-avoiding permutation, it must be of the form \[p = (k+1) (k+2) \cdots n \; k (k-1) (k-2) \cdots 1.\]
\end{proof}

To complete the proof of Theorem~\ref{george_ending_in_1}, we show these conditions are not only necessary, but also sufficient.

\begin{proposition}~\label{george_sufficient}  Let  $k \geq \frac{n}{2}$. Then any permutation $p = (k+1) (k+2) \cdots n \; k (k-1) (k-2) \cdots 1$  is a strongly 312-avoiding permutation.
\end{proposition}

\begin{proof}  It is clear that $p$ does not contain a 312 pattern.  
Calculating $p^2$, we see that it  is composed of the following three parts:
\begin{enumerate}
\item{The first $n-k$ entries will be $(n-k),\ldots,3,2,1$ respectively.}
\item{The next $k-(n-k)=2k-n$ entries correspond to entries with values and positions $k, (k-1), \ldots, n-k+1$ in $p$.  Because the order of these entries is monotone decreasing in $p$, these entries will be in increasing order in $p^2$.}
\item{The last $n-k$ entries will reverse the order of the first $n-k$ entries in $p$.  That is, the last $n-k$ entries of $p^2$ are $n, (n-1), \ldots, (k+1)$.}
\end{enumerate}

Hence $p^2 = (n-k)\; (n-k-1) \cdots 3\;2\;1 \; (n-k+1) \; (n-k+2) \; \cdots k \; n \; (n-1) \cdots (k+1)$, a concatenation of three blocks of monotone sequences in increasing order.  Thus $p^2$ avoids 312 as well.  As such, $p$ is a strongly 312-avoiding permutation.
\end{proof}

\subsection{Enumeration}
Rearranging \eqref{funceq}, we get the equality
\begin{equation} \label{sav} \textup{Sav}_{312}(z) =\frac{1}{1-B(z)}.\end{equation}
It follows from Theorem~\ref{george_ending_in_1} that there are $\lfloor \frac{n}{2} \rfloor$ strongly 312-avoiding permutations of length $n$ and ending in 1 if $n\geq 2$, and there is one such permutation if $n=1$.
Therefore, \[B(z)=z+\frac{(z+1)z^2}{(1-z^2)^2}=\frac{z^4-z^3+z}{(z-1)^2(z+1)}.\]
So \eqref{sav} yields
\[ \textup{Sav}_{312}(z) =\frac{-z^3+z^2+z-1}{z^4-2z^3+z^2+2z-1}.\]

So in particular,  $\textup{Sav}_{312}(z)$ is rational. Its root of smallest modulus is about 0.4689899435, so the exponential growth rate of the sequence of the numbers $\textup{Sav}_{312}(n)$ is the reciprocal of that root, or about 2.132241882. 
The first few elements of the sequence, starting with $n=1$, are 1, 2, 4, 9, 19, 41, 87, 186, 396, 845.

Interestingly, the sequence is in the Online Encyclopedia of Integer Sequences \cite{oeis} as Sequence A122584, where it is mentioned in connection to work in 
Quantum mechanics \cite{messiah}.

\subsection{Equivalent patterns}
As the pattern 231 is the inverse of 312, it is straightforward to see that $p$ is strongly 231-avoiding if and only if
$p^{-1}$ is strongly 312-avoiding. Therefore, $\textup{Sav}_{312}(n)=\textup{Sav}_{231}(n)$ for all $n$. 

Similarly, for any pattern $q$, we have the equality $\textup{Sav}_{q}(n)=\textup{Sav}_{q^{-1}}(n)$ for all $n$. The following result, while similar in flavor, is a little bit less obvious.

\begin{proposition} Let $q$ be any pattern, and let $q'$ denote the {\em reverse complement} of $q$. 
Then \[\textup{Sav}_{q}(n)=\textup{Sav}_{q'}(n)\] for all positive integers $n$.
\end{proposition}

\begin{proof} Let $p$ be a permutation of length $n$, and  let $r$ denote the involution $(1\ n)(2 \ n-1 )\cdots $.  Let us multiply permutations left to right.
Then $pr$ is the reverse of $p$, while $rp$ is the complement of $p$. Note that  $(rp)r = r(pr)$, showing why it does not matter in which order we  take the reverse and the complement of a permutation. 

Now the square of the reverse complement of $p$ is
\[(rpr) (rpr) =rprrpr=rp^2r,\] since $r$ is an involution, so $r^2$ is the identity. On the other hand, 
$rp^2 r$ is also the reverse complement of $p^2$. 

If $p$ is strongly $q$-avoiding, then in particular, $p$ is $q$-avoiding, so the reverse complement $p'$ of $p$ avoids $q'$.
Furthermore,  $p^2$ avoids $q$, so the reverse complement of $p^2$, {\em which we just proved is also
the square of the reverse complement of $p$} also avoids $q'$. So $p'$ is strongly $q'$-avoiding. 
\end{proof}

Note that taking the reverse complement of $p$ corresponds to rotating the permutation matrix of $p$ by
180 degrees, and that rotation commutes with matrix multiplication in the sense that $A'B'=(AB)'$, where $'$
denotes the 180 degree rotation. That implies that $p$ and $p^2$ avoid $q$, if and only if  $p'$ and $(p^2)'$ avoid $q'$. 

As a consequence,  $\textup{Sav}_{132}(n)=\textup{Sav}_{213}(n)$ for all $n$. This provides some additional motivation to compute
the numbers  $\textup{Sav}_{132}(n)$. We will discuss that very difficult task in Section~\ref{sec132}.

\section{The pattern 321} \label{sec-321}
It is straightforward to prove a lower bound for the numbers $\textup{Sav}_{321}(n)$ that shows that for large $n$,
the inequality  $\textup{Sav}_{321}(n) > \textup{Sav}_{312}(n)$ holds.

Indeed, let us call a permutation $p=p_1p_2\cdots p_n$
{\em block-cyclic} if it has both of the following properties.
\begin{enumerate}
\item It is possible to cut $p$ into blocks $B_1,B_2,\ldots ,B_t$ of entries in consecutive positions so that
for all $i<j$, the block $B_i$ is on the left of the block $B_j$, and each entry in $B_i$ is smaller than each entry in $B_j$.
\item Each block is either a singleton, or its entries can be written in one-line notation as $(a+i) \: (a+i+1) \cdots (a+k) \: (a+1) \cdots (a+i-1)$,
for some integers $1 < i \leq k$.  That is, each block is a singleton or a power of the cycle $(a +1 \; a+2 \cdots a+k)$ that is not the identity.
(It follows from Property 1 that the set of entries in each block is an {\em interval}.)
\end{enumerate}

\begin{example} The following are all block-cyclic (the bars are indicating the border between blocks).  
\begin{enumerate} \item $p=1|423|65|897$,
\item $p=4123|5|8967$, 
\item $p=1|2|3|4|5$. 
\end{enumerate}
\end{example}

Note that each block-cyclic permutation is 321-avoiding. Furthermore, {\em any} power of a block-cyclic permutation is
block-cyclic, and so it is also 321-avoiding. Therefore, block-cyclic permutations are all 321-avoiding. Let $h_n$ be the number of block-cyclic permutations of length $n$, and let $H(z)=\sum_{n\geq 0} h_nz^n$. 
The number of allowed blocks of size $k$ is 1 if $k=1$, and $k-1$ if $k>1$ (since longer blocks cannot be monotone
increasing), leading to the formula 
\[H(z)= \frac{1}{1-z-\sum_{k\geq 2} (k-1)z^k} =\frac{1}{1-z-\frac{z^2}{(1-z)^2}}=\frac{(1-z)^2}{1-3z+2z^2-z^3}.\]
The singularity of smallest modulus of the denominator is about 0.430159709, so the exponential growth rate
of the sequence $h_n$ is the reciprocal of that number, or about 2.324717957. As $h_n \leq \textup{Sav}_{321}(n)$
for all $n$, we have proved the following. 

\begin{corollary} The inequality \[2.3247 \leq \limsup_{n\rightarrow \infty}\sqrt[n]{\textup{Sav}_{321}(n)} \]
holds. 
\end{corollary}

Note that this proves that for large $n$, there are more permutations so that {\em all} their powers avoid 321 than
permutations so that just they and their square avoids 312. It is possible to improve this lower bound with more
complicated constructions, but we have no conjecture as to what the actual growth rate of the sequence
$\textup{Sav}_{321}(n)$ is.

\section{The pattern 132} \label{sec132}

We did not succeed in our efforts to enumerate strongly 132-avoiding permutations. On the one hand, the set of such 
permutations clearly contains all {\em involutions} that avoid 132, and the number of such involutions is known 
\cite{simion} to be
${n\choose \lfloor n/2 \rfloor}$. Therefore, the exponential growth rate of the sequence of the numbers $\textup{Sav}_{132}(n)$
is at least 2. On the  other hand, the first few terms of the sequence starting with $n=1$ are
 1, 2, 5, 12, 24, 50, 101, 202, 398, 806, 1568, 3148, 6198, 12306, 24223, 48314, which seems to suggest an 
exponential growth rate of {\em exactly} 2. Considering the first 40 terms, the trend holds. 

Therefore, we ask the following questions. 

\begin{question} Is it true that 
\[\lim_{n\rightarrow \infty} \sqrt[n]{\textup{Sav}_{132}(n)} =2 ?\]
\end{question}

Note that we do not even know that the limit in the previous question exists.  Indeed, let us call a pattern
$q$ {\em indecomposable} if it is not the direct sum of other patterns, that is, when it cannot be cut into 
two parts so that everything before the cut is {\em smaller} than everything after the cut. For instance, 
2413 and 3142 are indecomposable, but 132 is not. 

\begin{proposition} If $q$ is indecomposable, then $\lim_{n\rightarrow \infty} \sqrt[n]{\textup{Sav}_{q}(n)}$
exists. \end{proposition}

\begin{proof}
It is clear that if $p$ and $p'$ are both strongly $q$-avoiding, then so is their direct sum $p\oplus p'$. 
Therefore, the inequality $\textup{Sav}_n (q) \cdot \textup{Sav}_m(q) \leq \textup{Sav}_{n+m}(q) $ holds, so
by Fekete's lemma, the sequence $\sqrt[n]{\textup{Sav}_{q}(n)}$ is monotone increasing. As that sequence
is bounded from above, it has to be convergent. 
\end{proof} 

Perhaps this is the reason why 132, and its reverse complement 213, prove to be the most difficult of all patterns of length three for all purposes, because they are {\em not} indecomposable. All of 321, 231, and 312 are. The monotone pattern 123 is not. 

\begin{question} Is it true that for all positive integers $n$, the inequality
\[\textup{Sav}_{132}(n) \leq 2^n \] holds?
\end{question}

Note that if a permutation $p$ is of order three, that is, all its cycle lengths are 1 or 3, then $p^2=p^{-1}$, so among such
permutations, all 132-avoiding permutations are automatically strongly 132-avoiding. This raises the following question. 

\begin{question} How many 132-avoiding permutations of length $n$ are there in which each cycle length is 1 or 3?
\end{question}

\begin{center} {\bf Acknowledgment} \end{center}
We are grateful to Michael Cory, Michael Engen, Tony Guttmann and Jay Pantone for help in computing.

Mikl\'os B\'ona has been supported by Simons Collaboration Grant 421967.

\end{document}